\documentclass[a4paper,12pt]{amsart} 

\usepackage{hyperref}
\usepackage{amssymb, amsmath, amsthm}
\usepackage{amsfonts} 
\usepackage{color}
\usepackage{mathrsfs}
\usepackage{graphicx}
\usepackage{comment}
\usepackage{dcolumn}
\usepackage{bm} 
\usepackage[hang,small,bf]{caption}
\usepackage[subrefformat=parens]{subcaption}
\usepackage{algorithmic}
\usepackage{algorithm}
\numberwithin{equation}{section}

\newcommand{\e}{\varepsilon}
\renewcommand{\d}{\mathrm{d}}
\newcommand{\R}{\mathbb R}
\newcommand{\N}{\mathbb N}
\newcommand{\dv}{{\rm{div}}}

\newtheorem{thm}{Theorem}[section]
\newtheorem{lem}[thm]{Lemma}
\newtheorem{rmk}[thm]{Remark}

\numberwithin{equation}{section}
\newcommand{\Hm}[1]{\leavevmode{\marginpar{\tiny%
$\hbox to 0mm{\hspace*{-0.5mm}$\leftarrow$\hss}%
\vcenter{\vrule depth 0.1mm height 0.1mm width \the\marginparwidth}%
\hbox to 0mm{\hss$\rightarrow$\hspace*{-0.5mm}}$\\\relax\raggedright
#1}}}

\title[\tiny Approximation in an optimal design problem governed by the heat equation]{Approximation in
an optimal design problem\\ governed by the heat equation}

\author{Kei Matsushima}
\address[Kei Matsushima]{
Informatics and Data Science Program, 
Graduate School of Advanced Science and Engineering, 
Hiroshima University, 
Hiroshima 739-8521, 
JAPAN}
\email{matsushima@acs.hiroshima-u.ac.jp}

\author{Tomoyuki Oka}
\address[Tomoyuki Oka]{
Department of Intelligent Mechanical Engineering, Fukuoka Institute of Technology, Fukuoka 811-0295, JAPAN}
\email{t-oka@fit.ac.jp}

\date{\today}


\thanks{K.M. is partially supported by JSPS KAKENHI Grant Number 24K17191 and 23H03798.
T.O. is partially supported by JSPS KAKENHI Grant Number JP23K12997.
}

\begin{document}




\begin{abstract}
This paper studies a two-material optimal design problem for the time-averaged duality pairing between a (possibly time-dependent) heat source and the weak solution of an initial-boundary value problem for the heat equation with a two-material diffusion coefficient, under a volume constraint. In general, such optimal designs are not guaranteed to exist, and geometric constraints such as the perimeter are required. As an approximation of the problem with an additional perimeter constraint, a material representation based on a level set function, together with a perturbation of the Dirichlet energy, is employed. It is then shown that optimal level set functions exist for the perturbation problem, and the corresponding minimum value converges to that of the elliptic case, thereby elucidating the long-time behavior. Furthermore, two-material domains satisfying this property are also
constructed via the nonlinear diffusion-based level set method. In particular, the asymptotic behavior with respect to the perturbation parameter is clarified, and the validity of the approximation is established. 
\end{abstract}

\maketitle

\section{Introduction} 
Let $\Omega$ be a bounded domain of $\R^N$ with Lipschitz boundary $\partial\Omega$ for $N\ge 1$ and let $\Omega_i\subset\Omega$ ($i=0,1$) be such that $\Omega_0\cup\Omega_1=\Omega$ and $\Omega_0\cap\Omega_1=\emptyset$.
Assume that $\beta>\alpha>0$. 
The sets $\Omega_0$ and $\Omega_1$ are regarded as the regions occupied by the materials with diffusion coefficients $\alpha$ and $\beta$, respectively.
Thus the two-material diffusion coefficient matrix is described as $\kappa=\kappa[\chi]=(\alpha+(\beta-\alpha)\chi)\mathbb{I}$, where  $\mathbb{I}\in \R^{N\times N}$ is the identity matrix, and $\chi=\chi_{\Omega_1}\in L^{\infty}(\Omega;\{0,1\})$ is the characteristic function with respect to $\Omega_1$, which satisfies the following volume constraint\/{\rm :}
\begin{align*}
    \int_\Omega \chi(x)\, \d x=\gamma |\Omega|\quad \text{ for some } \gamma\in(0,1),
\end{align*}
and therefore, $\Omega_i\subset \Omega$ is represented by
$$
\Omega_i=\{x\in \Omega\colon \chi(x)=i\}.
$$
This paper focuses  on the optimal design of a two-material configuration that minimizes the following energy\/{\rm :}
\begin{align*}
    \mathcal{E}_T(\chi):=\frac{1}{T}\int_0^T\langle f(t), u(t)\rangle_{H^1_0(\Omega)}\, \d t
    \quad \text{ for } T>0,
\end{align*}
where $f\in L^2(0,T;H^{-1}(\Omega))$  
and $u=u_\chi\in V:=L^2(0,T;H^1_0(\Omega))\cap W^{1,2}(0,T;H^{-1}(\Omega))$ is a weak solution to the following heat equation\/{\rm :}
\begin{align}
\partial_t u-\dv(\kappa\nabla u)&=f\quad &&\text{ in } \Omega\times (0,T), \label{eq:1}\\
u&=0 &&\text{ on } \partial \Omega\times (0,T),\label{eq:2}\\
u&=u_0 &&\text{ in } \Omega\times \{0\}.
\label{eq:3}
\end{align}
Here $u_0$ stands for the initial data and lies on $L^2(\Omega)$ for simplicity
(see,~e.g.,~\cite[\S 11.1]{CD99} for the well-posedness).
Then the problem mentioned above is formulated as follows\/{\rm :}
\begin{align*}
\tag{CDP}
\qquad 
    \inf_{\chi\in \mathcal{CD}} \mathcal{E}_T(\chi),
\end{align*}
where $\mathcal{CD}=\{\chi\in L^{\infty}(\Omega;\{0,1\})\colon \|\chi\|_{L^1(\Omega)}=\gamma|\Omega|\}$.

\subsection{Known results}
As for the existence of solutions for (CDP), 
\emph{homogenization theory}, especially \emph{$H$-convergence theory} \cite{MT,BFF92,A02}, plays a crucial role, and it is well known that the classical design problem (CDP) is ill-posed since the minimizing sequence $(\chi_n)$ is bounded in $L^{\infty}(\Omega)$ but the density $\theta$, which is a weak-$\ast$ limit depending on the subsequence of $(n)$, is not in the classical design domain $\mathcal{CD}$ in general.
In the case of time-independent heat source $f=f(x)$, the dissipative energy (or the Dirichlet energy) is treated in \cite{MPP08}, and it is proved that the relaxed design problem, which is equivalent to the classical design problem, admits optimal densities and homogenized matrices.
The subsequent work \cite{AMP10} investigates the long-time behavior of the relaxed design problem (see also \cite{TZZ18}), and constructs optimized densities that attain the minimum value for the elliptic case.
Moreover, it is shown that the optimal designs are achieved by \emph{microstructures} in the form of sequential laminates of rank at most $N$, and that the order of lamination is reduced to at most $N - 1$ when $T>0$ is sufficiently large.

On the other hand, it remains a challenge to explicitly construct such fine-scale microstructures. In practical design, domains (or shapes) are usually chosen to avoid the intermediate set $[0<\theta<1]:=\{x\in \Omega:\,0<\theta(x)<1\}$, rather than relying on such microstructures. Typical approaches include \emph{SIMP methods}~\cite{BS03} and \emph{level set methods}~\cite{AJT04,HKO07}, which are well known. 
Indeed, several results have been reported in the engineering literature, such as \cite{zhuang2013topology,zhuang2014global,wu2019topology}. 
However, to the best of our knowledge, only a few results address the optimality of the obtained domains and their relation to the classical design problem. Indeed, even in the steady-state case (i.e.,~the elliptic case), establishing global optimality with respect to the domain is still a major open problem.

In the elliptic case, the energy functional $\mathcal{E}_T$ coincides with the dissipative energy through the weak form. 
Since the optimal homogenized matrix can be characterized as the rank-1 simple laminate, and the adjoint field appearing in the sensitivity (or the Fr\'echet derivative) is the same as the state itself, the relaxed design problem reduces to replacing the characteristic function by the density.
Furthermore, by the dual energy method (see \cite[\S 2.4]{ACMOY19}), the relaxed design problem 
admits only global solutions due to convexity. 
As a consequence, the optimal density can be constructed by employing the projected gradient descent method (see \cite[\S 5.2.3]{ACMOY19}).
Moreover, the optimal density is characterized in \cite{A02,CM70}, and in particular, it lies on $H^1(\Omega)$, e.g.,~when $\Omega$ is a ball and $f \in L^2(\Omega)$ is radial (see \cite{CD15}).

As in such a specific case, the additional regularity assumption is reasonable
to eliminate the intermediate set. Indeed, additional geometric assumptions are required to ensure the existence of solutions for (CDP) even in the elliptic case. In \cite{AB93}, the existence theorem has been proved by adding the perimeter constraint. 
From a functional analytic point of view, the compact embedding $BV(\Omega)\Subset  L^1(\Omega)$ plays a crucial role in characterizing the weak-$\ast$ limit of $(\chi_n)$; in particular, the limit remains a characteristic function.   
Therefore $L^\infty(\Omega)\cap BV(\Omega)$ provides a theoretically suitable framework.
Furthermore, the Modica--Mortola functional is known as a \emph{phase-field approximation} of the $BV$ framework, and it $\Gamma$-converges to the perimeter functional, thereby yielding the sharp interface limit (see e.g.,~\cite{M87,S88,D93,B02}). Hence, as in the smoothness results in  \cite{CD15}, the adoption of $L^\infty(\Omega)\cap H^1(\Omega)$ 
yields a numerically convenient and theoretically consistent framework for approximation.

Based on the results mentioned above, the so-called \emph{reaction-diffusion-based level set method} has been developed in \cite{T10}, 
which has the advantage that it allows for topological changes while decreasing the energy, independently of the topology of the initial domain. Moreover, since one only needs to solve the time-discrete equation in the fixed domain $\Omega$, no additional numerical techniques such as remeshing are required, which makes the implementation comparatively simple. 
However, in heat diffusion problems such as (CDP), the interface $\partial\Omega_0 \cap \partial\Omega_1$ often exhibits numerous spikes, and such spiky features cause it to oscillate during the optimization process, leading to convergence difficulties. In \cite{OY23},
this issue was addressed by extending the so-called \emph{nonlinear diffusion-based level set method} through the use of degenerate diffusion coefficients, and a numerically tractable scheme was also developed for heat diffusion problems.

In the previous work \cite{O24}, by replacing the characteristic function $\chi$ with the positive part of the level set function $\phi_+$ (defined in the next subsection), the problem setting proposed in \cite{OY23} was slightly modified in terms of the sensitivity, and existence of optimal level set functions was established. It is noteworthy that the energy associated with the optimal level set function decreases monotonically with respect to the perturbation parameter ($\e>0$) and converges to the minimum value under the smoothness assumption as in \cite{CD15}. Furthermore, by employing the method devised in \cite{O24}, domains that simultaneously satisfy both of these optimality conditions and exhibit no intermediate sets are obtained.
More generally, nonlinear boundary conditions described by the maximal monotone operator were also treated in \cite{KMO24}.

Consequently, extending the framework of \cite{O24} to the nonstationary setting will provide a basis for constructing two-material domains, on which an analysis of optimality can be carried out. 

This paper therefore contributes to the multiscale modeling community by bridging mathematical analysis, phase-field approximations and level set methods in the context of time-dependent heat diffusion design problems.
In particular, to the best of our knowledge, previous works do not address time-dependent heat sources, which serves as a motivation for the problem introduced in the next subsection.

\subsection{Main results}

This paper aims to construct an approximate solution to (CDP) with almost no intermediate sets by restricting the class to $H^1(\Omega)$.  
To this end, we shall introduce a material representation via the level set function given as
\begin{align*}
\phi(x)
\begin{cases}
>0,\quad& x\in \Omega_1,\\
=0,\quad& x\in \partial\Omega_1\cap \partial\Omega_0,\\
<0,\quad&x\in \Omega_0.
\end{cases}
\end{align*}
By employing the level set function, the two-material diffusion coefficient matrix $\kappa[\chi]$ is formally replaced with $\kappa[\phi_+^m]$ for $m\ge 1$ under $\|\phi_+\|_{L^{\infty}(\Omega)}\le 1$. 
Here $\phi_+$ denotes the positive part of $\phi$, i.e.,~$\phi_+=\max(\phi,0)$.
In order to eventually approximate $\kappa[\chi]$ by $\kappa[\phi_+^m]$ in the optimization procedure,
we shall also introduce the Dirichlet energy of the level set function as a perturbation term of the minimization problem (see $(P)$ below). 
Therefore the time-averaged energy $\mathcal{E}_T(\chi)$ will be approximately treated as
\begin{align*}
J_\e^T(\phi)
=
\frac{1}{T}\int_0^T \langle f(t), u_\phi(t)\rangle_{H^1_0(\Omega)}\, \d t
    +
    \frac{\e}{2}\int_\Omega |\nabla \phi(x)|^2\, \d x
=:\mathcal{E}_T(\phi_+^m)+\mathcal{I}(\phi),
\end{align*}
where $\varepsilon>0$ and $u_\phi\in V$ is a unique weak solution to \eqref{eq:1}--\eqref{eq:3} with $\kappa=\kappa[\phi_+^m]$, and then
our approximate problem is formulated as follows\/{\rm :}  
\begin{align*}
\tag{$P$}
\inf_{\phi\in U_{\rm ad}}
    J_\e^T(\phi),
\end{align*}
where 
$U_{\rm ad}=\{\phi\in L^{\infty}(\Omega)\cap H^1(\Omega)\colon  
\|\phi_+\|_{L^1(\Omega)}=\gamma |\Omega|,\  
|\phi|\le 1\}$.

Now, our theoretical result reads,
\begin{thm}\label{T:main}
Under the same assumptions mentioned above for $\Omega$, $f$ and $u_0$, 
let $u_\phi\in V$ be a weak solution to \eqref{eq:1}--\eqref{eq:3} with $\kappa=\kappa[\phi_+^m]$ for $m\ge 1$.
Then the following {\rm (i)--(iii)} are all satisfied\/{\rm :}
\begin{itemize}
\item[\rm (i)] {\rm(}Existence of solutions{\rm ).} For every $T>0$ and $\e>0$, there exists at least one solution of $(P)$.
\item[\rm (ii)] {\rm(}Long-time behavior{\rm ).}
For every $T>0$, let $\phi_T^\ast\in U_{\rm ad}$ be a solution to $(P)$.
Suppose that $f\in L^2_{\rm loc}([0,+\infty);H^{-1}(\Omega))$ satisfies   
\begin{align}\label{eq:f-assump}
\lim_{T\to+\infty}\frac{1}{T}\int_0^T\|f(t)-f_\infty\|_{H^{-1}(\Omega)}^2\, \d t=0
\end{align}
for some $f_\infty\in L^2(\Omega)$.
Then there exists a {\rm(}not relabeled{\rm)} subsequence of $(T)$ and
$\phi_\infty^\ast\in U_{\rm  ad}$ such that 
$$
\phi_T^\ast \to \phi_\infty^\ast \quad 
\text{ weakly-$\ast$ in } L^{\infty}(\Omega)
 \text{ and } 
\text{ weakly in } H^1(\Omega),
$$
and moreover, $\phi_\infty^\ast\in U_{\rm  ad}$ is a solution to the following problem for the elliptic case\/{\rm :}
\begin{align*}
\tag{$P_\infty$} 
    \min_{\phi\in U_{\rm ad}}\left\{J_\e^\infty(\phi):=\mathcal{E}_\infty(\phi_+^m)+\mathcal{I}(\phi)=
    \int_\Omega f_\infty(x)\overline{u}_\phi(x)\,\d x+\frac{\e}{2}\int_\Omega|\nabla\phi(x)|^2\, \d x
    \right\},
\end{align*}
where $\overline{u}_\phi\in H^1_0(\Omega)$ is a unique weak solution to 
\begin{align}\label{eq:elliptic}
    -\dv[\kappa\nabla \overline{u}_\phi]=f_\infty \quad \text{ in } H^{-1}(\Omega)
\end{align}
and $\kappa=\kappa[(\phi_\infty^\ast)_+^m]$.
In particular, $J_\e^S(\phi_S^\ast)\le J_\e^T(\phi_T^\ast)$ for $S\ge T$ under the assumption that $f$ and $u_0$ are some constants {\rm(}see \eqref{eq:f-u0-assump} below{\rm )}.
\item[\rm (iii)] {\rm(}Asymptotic behavior{\rm ).}
For every $\e>0$, let $\phi_\e^\ast\in U_{\rm ad}$ be a solution to $(P)$  and 
let $m=1$. Then there exists a {\rm(}not relabeled{\rm)} subsequence of $(\e)$ and $\phi^\ast_0\in U_{\rm ad}$ such that
$$
\phi^\ast_\e\to \phi^\ast_0\quad \text{ weakly-$\ast$ in } L^{\infty}(\Omega)
 \text{ and  weakly in } H^1(\Omega).
$$
Moreover, under the assumptions that 
$f$ and $u_0$ are constants such that $fu_0\ge 0$, 
it holds that
$$
    \mathcal{E}_T((\phi^\ast_0)_+)=\inf_{\chi\in \mathcal{CD}} \mathcal{E}_T(\chi),
$$
provided that $\theta^\ast\in H^1(\Omega)$. Here $\theta^\ast$ is the optimal density characterized by 
$
    \theta^\ast\in \text{\rm argmin}_{\theta\in\Theta}\, \mathcal{E}_T(\theta),
$
where
$\Theta=\{\theta\in L^{\infty}(\Omega;[0,1]) \colon \|\theta\|_{L^{1}(\Omega)}=\gamma|\Omega|\}$.
\end{itemize}
\end{thm}

\begin{rmk}[Interpretation of the assertions]\label{R:interpre}
\rm
As for the assertions (ii) and (iii) of Theorem \ref{T:main}, the following interpretation can be made\/{\rm :}
\begin{itemize}
\item[(i)]{\rm(}Optimality conditions{\rm )}
The assertion (ii) of Theorem \ref{T:main} can be regarded as optimality conditions for $(P)$. 
Thus, for every $T>0$, if one can construct the optimal level set function $\phi_T^\ast\in U_{\rm ad}$, then $J^T_\e(\phi_T^\ast)$ is monotonically decreasing with respect to $T>0$ and attains the minimum value for the elliptic case with $T>0$ large enough.

\item[(ii)]{\rm(}Approximate solution{\rm )}  
Furthermore, by virtue of the assertion (iii) of Theorem \ref{T:main}, 
$\phi_T^\ast\in U_{\rm ad}$ is also regarded as an approximate solution to (CDP) for $m=1$ and $\e>0$ small enough.
It should be noted that $\mathcal{E}_T(\phi_+)\le \mathcal{E}_T(\phi_+^m)$ for $m\ge 1$ by the sensitivity analysis (see \eqref{eq:positive} below) and $\phi_+\ge \phi_+^m$ for all $\phi\in U_{\rm ad}$. However, from the viewpoint of differentiability, the case $m>1$ is more tractable for numerical verification of optimality, and moreover, this monotonicity will provide a basis for the emergence of two-material domains (see \S\ref{S:numerical} below).
\end{itemize}
\end{rmk}

\subsection{
Structure of the paper
}

This paper is composed of five sections.
The next section proves (i) of Theorem \ref{T:main} by means of the direct method.
Section 3 is devoted to proving (ii) of Theorem \ref{T:main}, and in particular, time-dependent heat sources are allowed (cf.~\cite{MPP08,AMP10}). Moreover, the Galerkin method with the elliptic operator in \eqref{eq:1} is employed to demonstrate the time monotonicity of the energy, which plays a crucial role in the discussion of optimality.
Section 4 is devoted to the proof of (iii) in Theorem \ref{T:main}. Here the sensitivity of the relaxed design problem is rigorously derived, and its nonpositivity is demonstrated, which clarifies the relation between $(P)$ and (CDP) under the additional regularity assumption, thereby establishing the validity of the approximation problem. 
In particular, the Galerkin method is again employed to prove nonpositivity. However, this property is not guaranteed for the parabolic case unlike the elliptic case; therefore, although the technique is standard, establishing nonpositivity is not straightforward due to the loss of self-adjointness in the parabolic case. 
In the final section, the optimality conditions of Remark \ref{R:interpre} are verified via the nonlinear diffusion-based level set method 
(see Figures \ref{fig:result-constant_result} and \ref{fig:result-varying_result} in Section \ref{S:numerical}). Numerical results suggest that a sequence of $\phi_T^\ast$ asymptotically approaches  $\phi_\infty^\ast$ for sufficiently large $T>0$.

\section{Proof of (i) in Theorem \ref{T:main}}\label{S:pf1}
Based on the lower bounds,
$$
J_{\e}^T(\phi)\ge \frac{1}{T}\int_0^T\langle f(t),u_{\phi}(t)\rangle_{H^1_0(\Omega)}\, \d t\ge-\frac{1}{2T}\|u_0\|_{L^2(\Omega)}^2>-\infty,
$$
we first take a minimizing sequence $(\phi_n)$ in $U_{\rm ad}$. Thus it holds that
$$
\lim_{n\to+\infty} J_\e^T(\phi_n)=\inf_{\phi\in U_{\rm ad}} J_\e^T(\phi).
$$
Due to $\|\phi\|_{L^2(\Omega)}^2\le |\Omega|$ for all $\phi\in U_{\rm ad}$, we have
\begin{align*}
    -\frac{1}{2T}\|u_0\|_{L^2(\Omega)}^2+\frac{\e}{2}\|\phi\|_{H^1(\Omega)}^2-\frac{\e}{2}|\Omega|\le
     -\frac{1}{2T}\|u_0\|_{L^2(\Omega)}^2+\mathcal{I}(\phi)\le
    J_\e^T(\phi),
\end{align*}
which implies the coercivity, that is, $J_\e^T(\phi)\to +\infty$ as $\|\phi\|_{H^{1}(\Omega)}\to +\infty$.
Hence the boundedness of $(\phi_n)$ in $H^1(\Omega)\cap L^{\infty}(\Omega)$ is obtained, and therefore, 
there exists a (not relabeled) subsequence of $(n)$ and $\phi^\ast\in H^1(\Omega)\cap L^{\infty}(\Omega)$ such that
\begin{align}
\phi_n &\to \phi^\ast &&\quad\text{ weakly in }\ H^1(\Omega),\label{eq:phi-conv1}\\
\phi_n &\to \phi^\ast &&\quad\text{ weakly-$\ast$ in }\ L^{\infty}(\Omega),\\
\phi_n &\to \phi^\ast &&\quad\text{ strongly in }\ L^{q_1}(\Omega),\\
(\phi_n)_+^{q_2}&\to(\phi^\ast)_+^{q_2} &&\quad\text{ strongly in }\ L^{q_1}(\Omega)\label{eq:phi-conv4}
\end{align} 
for all $1\le q_1,q_2<+\infty$. Moreover, we see that $\phi^\ast\in U_{\rm ad}$, and in particular, for any $\Phi\in L^2(\Omega\times (0,T);\R^N)$,  
\begin{align}
\kappa[(\phi_n)_+^m]\Phi \to \kappa[(\phi^\ast)_+^m]\Phi \quad\text{ strongly in }\ L^2(\Omega\times (0,T);\R^N).
\label{eq:flux-conv}
\end{align} 

We next consider the weak limit of $u_{\phi_n}\in V$, which
is a weak solution to \eqref{eq:1}--\eqref{eq:3} with $\kappa=\kappa[(\phi_n)_+^m]$. To this end, we recall the following weak form\/{\rm :}
\begin{align*}
\langle \partial_t u_{\phi_n}(t), \psi\rangle_{H^1_0(\Omega)}+\int_\Omega \kappa[(\phi_n)_+^m](x)\nabla u_{\phi_n}(x,t)\cdot \nabla \psi(x)\, \d x=\langle f(t), \psi\rangle_{H^1_0(\Omega)} 
\end{align*}
for all $\psi\in H^1_0(\Omega)$ and a.e.~$t\in (0,T)$.
Testing it by $u_{\phi_n}\in V$, we have
\begin{align*}
\frac{1}{2}\frac{\d}{\d t}\|u_{\phi_n}(t)\|_{L^2(\Omega)}^2+\alpha\|u_{\phi_n}(t)\|_{H^1_0(\Omega)}^2
\le 
\langle f(t),u_{\phi_n}(t)\rangle_{H^1_0(\Omega)}
\le
\frac{1}{2\alpha}\|f(t)\|_{H^{-1}(\Omega)}^2
+
\frac{\alpha}{2}\|u_{\phi_n}(t)\|_{H^1_0(\Omega)}^2
\end{align*}
for a.e.~$t\in (0,T)$. Furthermore, for any $s\in(0,T]$, 
integrating it over $(0,s)$ yields 
\begin{align}
\label{eq:energyest2}
\sup_{t\in [0,T]}\|u_{\phi_n}(t)\|_{L^2(\Omega)}+\|u_{\phi_n}\|_{L^2(0,T;H^1_0(\Omega))}
\le
C(\|u_0\|_{L^2(\Omega)}+\|f\|_{L^2(0,T;H^{-1}(\Omega))}).
\end{align}
In particular, the boundedness of $(u_{\phi_n})$ in $W^{1,2}(0,T;H^{-1}(\Omega))$ is also obtained by combining that in $L^2(0,T;H^1_0(\Omega))$ with the inequality,  
\begin{align*}
\langle \partial_t u_{\phi_n}(t),\psi\rangle_{H^1_0(\Omega)}
\le 
(\beta\|u_{\phi_n}(t)\|_{H^1_0(\Omega)}
+\|f(t)\|_{H^{-1}(\Omega)})\|\psi\|_{H^1_0(\Omega)} 
\quad \text{for a.e.~$t\in (0,T)$}.
\end{align*}
Then there exists a (not relabeled) subsequence of $(n)$ and $u \in V$ such that
\begin{align}
u_{\phi_n} \to u \quad\text{ weakly in }\ V, 
\label{eq:convu}
\end{align}
which together with \eqref{eq:flux-conv} yields
\begin{align*}
\int_0^T\langle \partial_t u(t),\Psi(t)\rangle_{H^1_0(\Omega)}\,\d t
+
\int_0^T\int_\Omega \kappa[(\phi^\ast)_+^m](x)\nabla u(x,t)\cdot \nabla \Psi(x,t)\, \d x \d t
&=
\int_0^T\langle f(t), \Psi(t)\rangle_{H^1_0(\Omega)} \, \d t
\end{align*}
for all $\Psi=\psi_1(x)\psi_2(t)$, $\psi_1\in H^1_0(\Omega)$ and $\psi_2\in C^\infty_{\rm c}(0,T)$.
Thus $u\in V$ turns out to be a weak solution to \eqref{eq:1}--\eqref{eq:3} with $\kappa=\kappa[(\phi^\ast)_+^m]$.

We finally deduce from \eqref{eq:convu} and the lower semi-continuity of norm that
\begin{align*}
\inf_{\phi\in U_{\rm ad}}J_\e^T(\phi)\le 
J_\e^T(\phi^\ast)
\le
\lim_{n\to +\infty}
\mathcal{E}_T((\phi_n)_+^m)
+
\liminf_{n\to +\infty}\mathcal{I}(\phi_n)
=
\lim_{n\to+\infty}J_\e^T(\phi_n)
=
\inf_{\phi\in U_{\rm ad}}J_\e^T(\phi),
\end{align*}
which completes the proof.

\section{
Proof of (ii) in Theorem \ref{T:main}}
Let $u_T\in V$ and $\overline{u}_T\in H^1_0(\Omega)$ be weak solutions to \eqref{eq:1}--\eqref{eq:3} with $\kappa=\kappa[(\phi^\ast_T)_+^m]$ and \eqref{eq:elliptic} with $\kappa=\kappa[(\phi^\ast_T)_+^m]$, respectively.
We first show the boundedness of $(\phi_T^\ast)$ in $H^1(\Omega)$ with respect to $T>0$. 
By the definition of $\phi_T^\ast\in U_{\rm ad}$, we see that 
\begin{align}
\frac{\e}{2}\int_\Omega|\nabla\phi_T^\ast(x)|^2\, \d x
&\le
\frac{\e}{2}\int_\Omega|\nabla\phi_T^\ast(x)|^2\, \d x\nonumber\\
&\quad +
\frac{1}{T}\int_0^T\int_\Omega\kappa[(\phi_T^\ast)_+^m](x)\nabla u_T(x,t)\cdot \nabla u_T(x,t)\, \d x\d t+\frac{1}{2T}\|u_T(T)\|_{L^2(\Omega)}^2
\nonumber\\
&=
\frac{\e}{2}\int_\Omega|\nabla\phi_T^\ast(x)|^2\, \d x+
\frac{1}{T}\int_0^T\langle f(t), u_T(t)\rangle_{H^1_0(\Omega)}\d t+\frac{1}{2T}\|u_0\|_{L^2(\Omega)}^2\nonumber\\
&\le J_\e^T(\phi)+\frac{1}{2T}\|u_0\|_{L^2(\Omega)}^2
\nonumber
\end{align}
for all $\phi\in U_{\rm ad}$.
Furthermore, we deduce from \eqref{eq:energyest2} with \eqref{eq:f-assump} that
\begin{align*}
\lefteqn{\frac{1}{T}\int_0^T\langle f(t),u_\phi(t)\rangle_{H^1_0(\Omega)}\, \d t}\\
&\le 
\frac{1}{T}
\|f\|_{L^2(0,T;H^{-1}(\Omega))}\|u_\phi\|_{L^2(0,T;H^1_0(\Omega))}\\
&\le
\frac{C}{T}\|f\|_{L^2(0,T;H^{-1}(\Omega))}(\|u_0\|_{L^2(\Omega)}+\|f\|_{L^2(0,T;H^{-1}(\Omega))})
\le C\left(\frac{1}{\sqrt{T}}+1\right).
\end{align*}
Hence it holds that
$$
\frac{\e}{2}\int_\Omega|\nabla\phi_T^\ast(x)|^2\, \d x
\le
C_\e\left(\frac{1}{\sqrt{T}}+1\right)+\frac{1}{2T}\|u_0\|_{L^2(\Omega)}^2.
$$
Due to $|\phi|\le 1$ for all $\phi\in U_{\rm ad}$, the assertion holds true for $T>0$ large enough.
Hence there exist a (not relabeled) subsequence of $(T)$ and $\phi_\infty^\ast\in U_{\rm ad}$ such that
\eqref{eq:phi-conv1}--\eqref{eq:phi-conv4} with $\phi_n$ and $\phi^\ast$ being replaced by $\phi_T^\ast$ and $\phi_\infty^\ast$, respectively.

We next consider the limits of $\overline{u}_T\in H^1_0(\Omega)$ and $\mathcal{E}((\phi_T^\ast)_+^m)$. 
By noting that
$
\|\overline{u}_T\|_{H^1_0(\Omega)}\le \frac{1}{\alpha}\|f_\infty\|_{H^{-1}(\Omega)}
$,
there exist a (not relabeled) subsequence of $(T)$ and $\overline{u}_{\infty}\in H^1_0(\Omega)$ such that
\begin{align}
\overline{u}_T&\to \overline{u}_\infty\quad && \text{ weakly in }\ H^1_0(\Omega),\label{eq:Hconv-u}\\
\kappa[(\phi^\ast_T)_+^m]\nabla \overline{u}_T&\to \kappa[(\phi_\infty^\ast)_+^m]\nabla \overline{u}_\infty && \text{ weakly in }\ L^2(\Omega;\R^N), \nonumber
\end{align}
which implies that
$\overline{u}_\infty\in H^1_0(\Omega)$ becomes a weak solution to \eqref{eq:elliptic} with $\kappa=\kappa[(\phi_\infty^\ast)_+^m]$.
Then we claim that, up to a subsequence of $(T)$, 
\begin{align}
\lim_{T\to+\infty}  
\mathcal{E}_T((\phi_T^\ast)_+^m)
=
\langle f_\infty,\overline{u}_{\infty}\rangle_{H^1_0(\Omega)}.
\label{eq:conv-steady}
\end{align}
Indeed, since
$v_T:=u_T-\overline{u}_T$ is a weak solution to
$$
 \partial_t v_T -\dv[\kappa[(\phi_T^\ast)_+^m]\nabla v_T]=f-f_\infty \  \text{ in } \Omega\times (0,T),\quad v_T|_{\partial\Omega}=0,\quad v_T|_{t=0}=u_0-\overline{u}_T,
$$
we deduce from \eqref{eq:energyest2} and \eqref{eq:f-assump} that
\begin{align*}
\frac{1}{\sqrt{T}}\|v_{T}\|_{L^2(0,T;H^1_0(\Omega))}
\le
\frac{C}{\sqrt{T}}(\|u_0-\overline{u}_T\|_{L^2(\Omega)}+\|f-f_\infty\|_{L^2(0,T;H^{-1}(\Omega))})\le \frac{C}{\sqrt{T}},
\end{align*}
and in particular, for $T>0$ large enough,
\begin{align*}
\frac{1}{T}\|u_{T}\|_{L^2(0,T;H^1_0(\Omega))}
\le
\frac{C}{T}(\|u_0\|_{L^2(\Omega)}+\|f\|_{L^2(0,T;H^{-1}(\Omega))})\le \frac{C}{\sqrt{T}},
\end{align*}
which together with \eqref{eq:f-assump} yields
\begin{align}
\label{eq:conv-steady2}
\lefteqn{\left|  \mathcal{E}_T((\phi_T^\ast)_+^m)-\langle f_\infty,\overline{u}_T\rangle_{H^1_0(\Omega)}\right|}\\
 &\le
 \frac{1}{T}\left| \int_0^T\langle f(t)-f_\infty,u_T(t)\rangle_{H^1_0(\Omega)}\d t\right|
 +
 \frac{1}{T}\left| \int_0^T\langle f_{\infty},v_T(t)\rangle_{H^1_0(\Omega)}\d t\right|\displaybreak[0]
\nonumber\\
 &\le
 \frac{\|f-f_\infty\|_{L^2(0,T;H^{-1}(\Omega))}\|u_T\|_{L^2(0,T;H^1_0(\Omega))}}{T}
 +
 \frac{\|f_\infty\|_{H^{-1}(\Omega)}\|v_T\|_{L^2(0,T;H^1_0(\Omega))}}{\sqrt{T}}
\nonumber\\
&\le
C\left(
 \frac{\|f-f_\infty\|_{L^2(0,T;H^{-1}(\Omega))}}{\sqrt{T}}+\frac{1}{\sqrt{T}}
\right)
=:C_T.\nonumber
\end{align}
Combining it with \eqref{eq:Hconv-u}, we obtain \eqref{eq:conv-steady} by noting that
\begin{align*}
\left|  
\mathcal{E}_T((\phi_T^\ast)_+^m)
-
\langle f_\infty,\overline{u}_{\infty}\rangle_{H^1_0(\Omega)}
 \right|
 &\le
\left| \mathcal{E}_T((\phi_T^\ast)_+^m)
-
\langle f_\infty,\overline{u}_{T}\rangle_{H^1_0(\Omega)}\right|+
 \left| \langle f_\infty,\overline{u}_{T}-\overline{u}_{\infty}\rangle_{H^1_0(\Omega)}\right|\displaybreak[0]
\\
&\le
C_T
+
\|f_\infty\|_{L^2(\Omega)}\|\overline{u}_T-\overline{u}_\infty\|_{L^2(\Omega)}
\to 0\quad \text{ as }\ T\to+\infty.
\end{align*}
Here
note that the assumption $f_\infty\in L^2(\Omega)$ is essential in the last line. 

We next establish the optimality of $\phi_\infty^\ast\in U_{\rm  ad}$ for $(P_\infty)$. 
If $\phi_\infty^\ast\in U_{\rm  ad}$ is not a solution to $(P_\infty)$,
then there exist $\eta>0$ and another solution $\phi_\eta\in U_{\rm ad}$ such that
\begin{align}
J_\e^\infty(\phi_\eta)=J_\e^\infty(\phi_\infty^\ast)
    -\eta.
    \label{eq:cont-1}
\end{align}
Let $u_{\phi_\eta}\in V$ and $\overline{u}_{\phi_\eta}\in H^1_0(\Omega)$ be  unique weak solutions to \eqref{eq:1}--\eqref{eq:3} with $\kappa=\kappa[(\phi_\eta)_+^m]$ and \eqref{eq:elliptic} with $\kappa=\kappa[(\phi_\eta)_+^m]$, respectively. As in \eqref{eq:conv-steady2}, we have
\begin{align*}
    \lim_{T\to +\infty} \mathcal{E}_T((\phi_\eta)_+^m)
    =
    \int_\Omega f_\infty(x)\overline{u}_{\phi_\eta}(x)\,\d x,
\end{align*}
which implies that there exists $N_0\in \N$ such that, for any $S\ge N_0$, 
\begin{align}
\mathcal{E}_S((\phi_\eta)_+^m)
    <
    \int_\Omega f_\infty(x)\overline{u}_{\phi_\eta}(x)\,\d x+\frac{\eta}{2}.
    \label{eq:cont-2}
\end{align}
Furthermore, by \eqref{eq:conv-steady}, there exists $N_1\in \N$ such that, for all $S\ge \max(N_0,N_1)$, 
\begin{align}
        \int_\Omega f_\infty(x)\overline{u}_{\infty}(x)\, \d x
    <
    \mathcal{E}_S((\phi_S^\ast)_+^m)
    +\frac{\eta}{2}.  
    \label{eq:cont-3}
\end{align}
Hence we conclude that
\begin{eqnarray*}
J_\e^S(\phi_\eta)
&\stackrel{\eqref{eq:cont-2}}{<}&
\int_\Omega f_\infty(x)\overline{u}_{\phi_\eta}(x)\,\d x
+\frac{\e}{2}\int_\Omega|\nabla \phi_\eta(x)|^2\, \d x
+\frac{\eta}{2}\\
&\stackrel{\eqref{eq:cont-1}}{\le}&
\int_\Omega f_\infty(x)\overline{u}_\infty(x)\,\d x
+\frac{\e}{2}\int_\Omega|\nabla \phi^\ast_S(x)|^2\, \d x
-\frac{\eta}{2}
\stackrel{\eqref{eq:cont-3}}{<}
J_\e^S(\phi_S^\ast),
\end{eqnarray*}
which contradicts the fact that $\phi_S^\ast\in U_{\rm ad }$ is a solution to $(P)$ with $T>0$ being replaced by $S>0$. Thus $\phi_\infty^\ast\in U_{\rm  ad}$ turns out to be a solution to $(P_\infty)$.

The rest of the proof is to show the monotonicity, that is, 
$J_\e^S(\phi_S^\ast)\le J_\e^T(\phi_T^\ast)$ for $S\ge T$. 
Due to $J_\e^S(\phi_S^\ast)\le J_\e^S(\phi_T^\ast)$ by the definition of $\phi_S^\ast\in U_{\rm ad}$,  it suffices to show $\mathcal{E}_S((\phi_T^\ast)_+^m)\le\mathcal{E}_T((\phi_T^\ast)_+^m)$. To this end, Galerkin's method via the eigenvalue problem for the elliptic operator $L(\cdot):=-  \dv(\kappa\nabla\cdot)$ with $\kappa=\kappa[(\phi_T^\ast)_+^m]$ is employed below. 
Let $(\lambda_k, w_k)\in \R_{>0}\times H^1_0(\Omega)$ be the $k$-th eigen pair of $L(\cdot)$ such that $0<\lambda_1\le \lambda_2\le \ldots$ and $\lambda_k\to +\infty$ as $k\to+\infty$. Then one can construct 
$u_\ell\in W^{1,2}(0,\sigma;H^{-1}(\Omega))\cap L^2(0,\sigma;H^1_0(\Omega))$ such that
\begin{align}
u_\ell\to u_{T}\quad \text{ weakly in } W^{1,2}(0,\sigma;H^{-1}(\Omega))\cap L^2(0,\sigma;H^1_0(\Omega))
\label{eq:conv-uell}    
\end{align}
for all $\sigma\in [T,+\infty)$, and $u_\ell$ is characterized as $u_\ell(x,t)=\sum_{k=1}^\ell w_k(x)d_k^\ell(t)$, where
$d_k^\ell\in C^1(0,\sigma)$ is a solution to 
$$
\begin{cases}
(d_k^\ell)'+\lambda_k d_k^\ell=F_k\quad \text{ in } (0,\sigma),\\
d^\ell_k(0)=(u_0, w_k)_{L^2(\Omega)}
\end{cases}
$$
and $F_k(t)=\langle f(t),w_k\rangle_{H^1_0(\Omega)}$.
Thus the time independence of $f$ yields 
$$
d_k^\ell(t)
=
\left(d^\ell_k(0)+\left(\int_0^t e^{\lambda_ks}F_k(s)\, \d s\right)\right)e^{-\lambda_k t}
=
\left(d_k^\ell(0)-\frac{F_k}{\lambda_k}\right)e^{-\lambda_kt}+\frac{F_k}{\lambda_k}.
$$
In addition, if $f$ and $u_0$ satisfy
\begin{align}
    f\left(u_0-\frac{f}{\lambda_1}\right) \ge 0,
    \label{eq:f-u0-assump}
\end{align}
then we infer that
\begin{align*}
\lefteqn{\int_0^S\int_\Omega f(x,t)u_\ell(x,t)\, \d x\d t}\\
&=\sum_{k=1}^\ell
\left[ f\left(u_0-\frac{f}{\lambda_k}\right)\left( \int_\Omega w_k(x)\, \d x\right)^2\int_0^Se^{-\lambda_kt} \, \d t
+
S\int_\Omega f w_k(x)\frac{F_k}{\lambda_k}\, \d x
\right],
\end{align*}
which together with the monotonicity of $H(S):=\|e^{-\lambda_kt}\|_{L^1(0,S)}/S$ yields
\begin{align*}
\frac{1}{S}\int_0^S\int_\Omega f(x,t)u_\ell(x,t)\, \d x\d t
\le
\frac{1}{T}\int_0^T\int_\Omega f(x,t)u_\ell(x,t)\, \d x\d t.
\end{align*}
Hence we see by \eqref{eq:conv-uell} that $\mathcal{E}_S((\phi_T^\ast)_+^m)\le\mathcal{E}_T((\phi_T^\ast)_+^m)$. This completes the proof of (ii).

\section{
Proof of (iii) in Theorem \ref{T:main}}
We first note that $(\phi_\e^\ast)$ is bounded in $H^1_0(\Omega)$.
Indeed, $a_\e:=J_\e^T(\phi_\e^\ast)$ is monotonically decreasing as $\e\to 0_+$ and $a_\e\ge \inf_{\chi\in \mathcal{CD}} \mathcal{E}_T(\chi)$. 
Thus $a_\e$ has a limit. By the coercivity, that is,  $J_\e^T(\phi)\to +\infty$ as $\|\phi\|_{H^1(\Omega)}\to +\infty$, the boundedness in $H^1(\Omega)$ is obtained. 
Hence there exist a (not relabeled) subsequence of $(\e)$ and $\phi_0^\ast\in H^1(\Omega)\cap L^{\infty}(\Omega)$ such that
\eqref{eq:phi-conv1}--\eqref{eq:phi-conv4} with $\phi_n$ and $\phi^\ast$ being replaced by $\phi_\e^\ast$ and $\phi_0^\ast$, respectively. Thus $\phi_0^\ast$ lies on $U_{\rm ad}$, and in particular, $(\phi_0^\ast)_+\in\Theta$ due to $m=1$.
By the same argument mentioned in the proof of (i) and the boundedness of $(\phi_\e^\ast)$ in $H^1(\Omega)$, we have
$
\lim_{\e\to 0_+} J_\e^T(\phi_\e^\ast)=\mathcal{E}_T((\phi_0^\ast)_+).
$
Since it follows that, for any $\phi\in U_{\rm ad}$, $J_\e^T(\phi_\e^\ast)\le J_\e^T(\phi)\to \mathcal{E}_T(\phi_+)$ as $\e\to 0_+$, we obtain $\mathcal{E}_T((\phi_0^\ast)_+)\le \mathcal{E}_T(\phi_+)$, and hence,  
$$
\mathcal{E}_T((\phi_0^\ast)_+)=
\min_{\theta\in \Theta\cap H^1(\Omega)}\mathcal{E}_T(\theta).
$$

The rest is to prove that  
$$
  \inf_{\chi\in \mathcal{CD}} \mathcal{E}_T(\chi)=  \min_{\theta\in \Theta\cap H^1(\Omega)} \mathcal{E}_T(\theta).
$$
By the same argument mentioned in \cite[Theorem 2.4]{MPP08}, it holds that
\begin{align*}
\inf_{\chi\in \mathcal{CD}} \mathcal{E}_T(\chi)
=
     \min_{(\theta,A)\in \mathcal{RD}} \mathcal{E}_T^{\rm hom}(\theta,A),
\end{align*}
where $\mathcal{RD}$ is the relaxed design domain given by 
$$
\mathcal{RD}=\{(\theta,A)\in G_\theta\colon \|\theta\|_{L^1(\Omega)}=\gamma|\Omega|\},
$$
$G_{\theta}$ is the set of $(\theta, A)\in L^{\infty}(\Omega;[0,1]\times \mathbb{S}^{N\times N})$ such that 
there exists $(\chi_n,\kappa[\chi_n])\in L^{\infty}(\Omega;\{0,1\}\times \mathbb{S}^{N\times N})$ such that  
$$
\chi_n\to \theta \text{ weakly-$\ast$ in }\ L^{\infty}(\Omega),
\quad 
\|\theta\|_{L^1(\Omega)}=\gamma|\Omega|,
\quad  
\kappa[\chi_n]\overset{H}{\to } A
$$
(see \cite{MT} for the notation ``$\overset{H}{\to}$"), 
$\mathcal{E}_T^{\rm hom}$ is the homogenized energy given by
    $$
    \mathcal{E}_T^{\rm hom}(\theta,A)=\frac{1}{T}\int_0^T\langle f(t), u(t)\rangle_{H^1_0(\Omega)}\, \d t,
    $$
and $u\in V$ is a weak solution to \eqref{eq:1}--\eqref{eq:3} with $\kappa$ being replaced by $A$. 
Thus it suffices to prove that
\begin{align}
   \min_{(\theta,A)\in \mathcal{RD}} \mathcal{E}_T^{\rm hom}(\theta,A)
   =
    \min_{\theta\in \Theta } \mathcal{E}_T(\theta).
     \label{eq:key}
\end{align}
To this end, we derive the Fr\'echet derivative of 
$\mathcal{E}_{\rm hom}(A):=\mathcal{E}_T^{\rm hom}(\theta, A)$ with respect to $A\in L^{\infty}(\Omega;\mathbb{S}^{N\times N})$ below.
Define $\mathcal{L}:L^{\infty}(\Omega;\mathbb{S}^{N\times N})\times  V\times V\to\R$ by 
\begin{align*}
\mathcal{L}(A, u, v)
&=\mathcal{E}_{\rm hom}(A)
+
\int_0^T\langle \partial_t u(t), v(t)\rangle_{H^1_0(\Omega)}\, \d t\\
&\quad +
\int_0^T\int_\Omega A(x)\nabla u(x,t)\cdot \nabla v(x,t)\, \d x\d t
-
\int_0^T\langle f(t),v(t)\rangle_{H^1_0(\Omega)}\d t.
\end{align*}
Differentiating it by $u\in V$, we have
\begin{align*}
\langle \partial_{u} \mathcal{L}(A, u, v), \psi \rangle_{V}
&=
\frac{1}{T}\int_0^T\langle f(t), \psi(t)\rangle_{H^1_0(\Omega)}\, \d t
+
\int_\Omega \psi(x,T)v(x,T)\, \d x
-
\int_\Omega \psi(x,0)v(x,0)\, \d x
\\
&\quad
-
\int_0^T\langle \partial_t v(t), \psi(t)\rangle_{H^1_0(\Omega)}\, \d t
+
\int_0^T\int_\Omega A(x)\nabla v(x,t)\cdot \nabla \psi(x,t)\, \d x\d t.
\end{align*} 
Since $u_0$ is a constant, $w:=(u-u_0)/T$ turns out to be a weak solution to the following equation\/{\rm :} 
$$
\partial_t w-\dv[A \nabla w]=f/T\ \text{ in } \Omega\times (0,T),\quad w|_{\partial\Omega}=-u_0/T,\quad w|_{t=0}=0.
$$
Furthermore, $p(x,t):=w(x,T-t)$\footnote{
Additional assumptions for $u_0$ and $f$ ensure that the adjoint field $p\in V$ can be written as the state field $u\in V$ in the sense of the weak solution. In general,  the adjoint field is characterized as a weak solution to the following final value problem\/{\rm :}  
\begin{align*}
-\partial_t p-\dv[A\nabla p]=(f/T)
\ \text{ in } \Omega\times (0,T),\quad
p|_{\partial\Omega}=0,\quad
p|_{t=T}=0.
\end{align*}
From the perspective of numerical computation, the final value problem can be solved in an analogous manner to the original initial value problem. The assumption allows us to reduce the cost of this numerical computation.
}  
satisfies
$$
-\langle \partial_t p(t), \varphi\rangle_{H^1_0(\Omega)}+\int_\Omega A(x)\nabla p(x,t)\cdot \nabla \varphi(x)\, \d x=\frac{1}{T}\left\langle f(T-t),\varphi\right\rangle_{H^1_0(\Omega)}\quad \text{ for all $\varphi\in H^1_0(\Omega)$}
$$
and a.e.~$t\in (0,T)$.
If $\psi(x,0)=0$ (see Lemma \ref{L:u'}), then we see by $f(x,t)=f(x,T-t)$ that
\begin{align}
\langle \partial_{u} \mathcal{L}(A, u, -p), \psi \rangle_{V}
=0.
\label{eq:f-sym}
\end{align} 
Since $\langle \partial_v\mathcal{L}(A,u,v), \psi\rangle_V=0$ for all $v\in V$ due to the weak form of $u\in V$, we deduce from Lemma \ref{L:u'} that 
\begin{align}
\label{eq:frechet}
\lefteqn{\langle \mathcal{E}_{\rm hom}'(A),h \rangle_{L^{\infty}(\Omega)}}\\
&=
\langle \partial_A\mathcal{L}(A, u, -p), h \rangle_{L^{\infty}(\Omega)}
+
\langle \partial_{u}\mathcal{L}(A, u, -p), u'h \rangle_{V}
+
\langle \partial_{-p}\mathcal{L}(A, u, -p), -p'h \rangle_{V}\nonumber\\
&=
-\frac{1}{T}\int_0^T\int_\Omega h(x)\nabla u(x,t)\cdot \nabla u(x,T-t)\, \d x\d t
\quad 
\text{ for all $h\in L^{\infty}(\Omega;\mathbb{S}^{N\times N})$. }\nonumber
\end{align}
Furthermore, it holds that
\begin{align}
\langle \mathcal{E}_{\rm hom}'(A),h \rangle_{L^{\infty}(\Omega)}\le 0
   \label{eq:positive}
\end{align}
for all $h\in L^{\infty}(\Omega;\mathbb{S}^{N\times N})$ satisfying the uniform ellipticity. Indeed,
as in the proof of (ii),
let $(\lambda_k, w_k)$ be the $k$-th eigen pair of $L(\cdot)$ with $\kappa=A$ and let $u_\ell\in V$ be a function that appeared in the proof of (ii). 
Then we deduce from the orthogonality of $H^1_0(\Omega)$ that
\begin{align*}
\lefteqn{\int_0^T\int_\Omega h(x)
\nabla u_\ell(x,t)\cdot\nabla u_\ell(x,T-t)\, \d x\d t}\\
&=
\sum_{k=1}^\ell
\left(\int_0^T
d_k^\ell(t)d_k^\ell(T-t)\, \d t \right) 
\underbrace{
\left(\int_\Omega h(x)\nabla w_k(x)\cdot  \nabla w_k(x)\, \d x\right)}_{\ge 0\ \text{ by the uniform ellipticity}},
\end{align*}
and moreover, the nonnegativity of $d^\ell_k(t)d^\ell_k(T-t)$ is obtained by noting that
\begin{align*}
d^\ell_k(t)d^\ell_k(T-t)
&=
\left[\left(d_k^\ell(0)-\frac{F_k}{\lambda_k}\right)e^{-\lambda_kt}+\frac{F_k}{\lambda_k}\right]
\left[\left(d_k^\ell(0)-\frac{F_k}{\lambda_k}\right)e^{-\lambda_k(T-t)}+\frac{F_k}{\lambda_k}\right]
\displaybreak[0]\\
&=
(d_k^\ell(0))^2e^{-\lambda_kT}+\underbrace{d_k^\ell(0)\frac{F_k}{\lambda_k}}_{\ge 0\ \text{ by } fu_0\ge 0}\underbrace{(-2e^{-\lambda_kT}+e^{-\lambda_kt}+e^{-\lambda_k(T-t)})}_{\ge 0}\displaybreak[0]\\
&\qquad +
\frac{F_k^2}{\lambda_k^2}\underbrace{(e^{-\lambda_kT}-
e^{-\lambda_kt}-e^{-\lambda_k(T-t)}
+1)}_{\ge 0}\ge 0.
\end{align*}
We next consider the limit of $u_\ell\in V$ as $\ell\to+\infty$. Recalling that
 \begin{align}
\langle\partial_t u_\ell(t), \varphi\rangle_{H^1_0(\Omega)}+(A\nabla u_\ell(t), \nabla \varphi)_{L^2(\Omega)}=\langle f(t),\varphi\rangle_{H^1_0(\Omega)}\quad \text{ for all $\varphi\in H^1_0(\Omega)$}
\label{eq:weak-form-um}    
\end{align}
and a.e.~$t\in (0,T)$ and testing \eqref{eq:weak-form-um} by $\partial_t u_\ell\in V$, we have
$$
\|\partial_t u_\ell(t)\|_{L^2(\Omega)}^2+\frac{1}{2}\frac{\d}{\d t}\int_\Omega A(x)\nabla u_\ell(x,t)\cdot \nabla u_\ell(x,t)\, \d x=\langle f(t), \partial_t u_\ell(t)\rangle_{H^1_0(\Omega)},
$$
and hence,  
\begin{align*}
\|\partial_t u_\ell\|_{L^2(\Omega\times (0,T))}+\| u_\ell\|_{L^{\infty}(0,T;H^1_0(\Omega))}
\le C(\|f\|_{L^2(\Omega\times (0,T))}+\|\nabla u_0\|_{L^2(\Omega)}).
\end{align*}
Here we used the fact that
\begin{align}
\underline{A}\xi\cdot\xi\le A(x)\xi\cdot\xi\le \overline{A}\xi\cdot\xi
\quad \text{ for all } \xi\in \R^N \text{ and a.e.~}x\in\Omega,
\label{eq:unif-ellip-hom}
\end{align}
where $\underline{A}$ and $\overline{A}$ are the inverse of the weak-$\ast$ limit of $\kappa[\chi_n]^{-1}$ and the weak-$\ast$ limit of $\kappa[\chi_n]$, respectively (see, e.g.,~\cite[Theorem
1.3.14]{A02}).
Furthermore, we claim that
\begin{align*}
    u_\ell \to u \quad \text{ strongly in }\ L^2(0,T;H^1_0(\Omega)). 
\end{align*}
Indeed, we deduce from the weak forms of $u_\ell$ and $u$,  the boundedness of $u_\ell$ in $L^2(0,T;H^1_0(\Omega))\cap W^{1,2}(0,T;L^2(\Omega))$ and the Aubin--Lions--Simon lemma \cite{S87} that
\begin{align*}
\lim_{\ell\to+\infty}\int_0^T\int_\Omega A(x)\nabla u_\ell(x,t)\cdot \nabla u_\ell(x,t)\, \d x\d t=
\int_0^T\int_\Omega A(x)\nabla u(x,t)\cdot \nabla u(x,t)\, \d x\d t.
\end{align*}
Let $r>0$ be such that $(A-r\mathbb{I})$ satisfies the uniform ellipticity. Then 
the lower semi-continuity of norm ensures that
\begin{align*}
\lefteqn{\liminf_{\ell\to +\infty} \int_0^T\int_\Omega (A(x)-r\mathbb{I})\nabla u_\ell(x,t)\cdot \nabla u_\ell(x,t)\, \d x\d t}\\
&\ge
\int_0^T\int_\Omega (A(x)-r\mathbb{I})\nabla u(x,t)\cdot \nabla u(x,t)\, \d x\d t,
\end{align*}
which yields
$$
\limsup_{\ell\to +\infty} \int_0^T\int_\Omega |\nabla u_\ell(x,t)|^2\, \d x\d t\le
\int_0^T\int_\Omega |\nabla u(x,t)|^2\, \d x\d t.
$$
Hence the assertion holds true, and then
\eqref{eq:positive} is obtained by noting that
\begin{align*}
0
&\le
\frac{1}{T}\limsup_{\ell\to +\infty}\int_0^T\int_\Omega h(x)\nabla u_\ell(x,t)\cdot \nabla u_\ell(x,T-t)\, \d x\d t \\
&=
\frac{1}{T}
\int_0^T\int_\Omega h(x)\nabla u(x,t)\cdot \nabla u(x,T-t)\, \d x\d t.
\end{align*}
Combining \eqref{eq:positive} with \eqref{eq:unif-ellip-hom} yields 
\begin{align*}
    \mathcal{E}_T(\theta)
    \le \mathcal{E}_T^{\rm hom}(\theta,
    A),
\end{align*}
which implies \eqref{eq:key} due to $(\theta, \kappa[\theta])\in\mathcal{RD}$. This completes the proof.

\section{Confirmation for the optimality conditions}\label{S:numerical}
Before presenting the numerical results for the optimal designs, we briefly touch on our numerical algorithm via the nonlinear diffusion-based level set method, which is based on the following 
adaptive (or generalized) forward–backward splitting scheme with $\zeta(\cdot)\ge 0$\/{\rm :}
\begin{align*}
\phi_{i+1}=\psi_{i+1}-\zeta(\phi_{i})\mathcal{I}'(\phi_{i+1}),\quad 
    \psi_{i+1}=\phi_i-\zeta(\phi_i) \mathcal{E}_T'((\phi_i)_+^m).
\end{align*}
As for the Fre\'chet derivative of $\mathcal{E}_T$ at $\phi_i\in U_{\rm ad}$, we see by \eqref{eq:frechet} that
\begin{align*}
    \lefteqn{\langle \mathcal{E}_T'((\phi_i)_+^m), h\rangle_{L^{\infty}(\Omega)}}\\
    &=
    \langle 
    \partial_{(\phi_i)_+^m}\mathcal{E}_T((\phi_i)_+^m), ((\phi_i)_+^m)' h\rangle_{L^{\infty}(\Omega)}\\
    &=
    \int_\Omega \left[-\frac{m(\beta-\alpha)}{T}\int_0^T|\phi_i(x)|^{m-1}\chi_{\phi_i}(x)
    \nabla u_{\phi_i}(x,t)\cdot \nabla u_{\phi_i}(x,T-t)\, \d t\right]h(x)\, \d x.
\end{align*}
Here $\chi_{\phi_i}\in L^{\infty}(\Omega;\{0,1\})$ stands for the characteristic function with respect to $[\phi>0]:=\{x\in\Omega\colon \phi(x)>0\}$.
Setting $\zeta(\phi_i)=\tau|\phi_i|^{1-q}/q$ for $q\in (0,1)$, one can also select $m=1$ (see also \cite{LZ11} for $m>1$ and $q=1$), and then the scheme mentioned above is rewritten by
\begin{align}\label{eq:nld}
    q|\phi_i|^{q-1}\frac{\phi_{i+1}-\phi_i}{\tau}-\e\Delta \phi_{i+1}
    = 
    \frac{m(\beta-\alpha)}{T}|\phi_i|^{m-1}\chi_{\phi_i}\int_0^T
    \nabla u_{\phi_i}(\cdot,t)\cdot \nabla u_{\phi_i}(\cdot,T-t)\, \d t
\end{align}
for $\tau>0$ small enough. 
In particular, \eqref{eq:nld} is regarded as the time-discretized version of the following nonlinear diffusion equation\/{\rm :} 
\begin{align*}
\partial_t \phi^{q}-\e\Delta \phi=
 \frac{m(\beta-\alpha)}{T}|\phi_i|^{m-1}\chi_{\phi_i}\int_0^T
    \nabla u_{\phi_i}(\cdot,t)\cdot \nabla u_{\phi_i}(\cdot,T-t)\, \d t
\end{align*}
as $q\approx 1$ and $\phi_i(x)=\phi(x,\tau i)$ formally.

Now, the following algorithm is proposed\/{\rm :}
\begin{algorithm}[H]
    \caption{Optimization for the level set function.}
    \label{alg}
    \begin{algorithmic}[1]
    \STATE 
    Let $i=0$.
    Set $\Omega\subset \R^N$, $\alpha,\beta, \gamma>0$, $f\in L^2(0,T;H^{-1}(\Omega;\R_+))$, $u_0\in L^2(\Omega;\R_+)$ and $\phi_0\in U_{\rm ad}$. 
    \STATE
    Solve \eqref{eq:1}--\eqref{eq:3} with  $\kappa=\kappa[(\phi_i)_+^m]$ to determine $u_{\phi_{i}}$ in \eqref{eq:nld}. 
    \STATE
    Compute $\phi_{i+1}$ from \eqref{eq:nld}.  
  \STATE
Determine $\lambda\in\R$ such that 
$
|\gamma|\Omega|-\|(\phi_{i+1}^\lambda)_+\|_{L^1(\Omega)}|\le \eta_1. 
$
Here
$
\phi^\lambda_{i+1}=\max\{-1,\min\{\phi_{i+1}+\lambda,1\}\}.
$
\STATE
   Check for the convergence condition 
   $
   \|\phi^\lambda_{i+1}-\phi_{i}\|_{L^1(\Omega)}\le \eta_2.
   $ 
   If it is satisfied, then terminate the optimization as $\phi_{i+1} \leftarrow \phi^\lambda_{i+1}$; otherwise, return 2 after setting $\phi_{i} \leftarrow \phi^\lambda_{i+1}$. 
    \end{algorithmic}
\end{algorithm}
Based on Algorithm \ref{alg}, we construct candidate solutions to $(P)$ numerically. 
Here we use a finite element method implemented on FreeFEM++ \cite{H12}. Throughout this section, we set $N=2$, $\Omega=(0,1)^2$, $u_0 \equiv 1$, $\alpha=1$, $\beta=10$, $\varepsilon=10^{-4}$ and $\phi_0\equiv\gamma=0.5$. The parameter $m$ is set to $m=3$ for the differentiability of $\mathcal E_T(\phi_+^m)$. In the finite-element analysis for \eqref{eq:1}--\eqref{eq:3}, we adopt the standard backward Euler scheme and piecewise linear ($\mathbb P_1$) elements over a triangular mesh. The convergence condition is set to $\eta_2 = 8\times 10^{-5}$.

Let us first consider the simplest case, where the source $f$ is constant in space and time with $f\equiv 10$. The proposed algorithm (Algorithm 1) is performed for the optimal design problems ($P_\infty$) (elliptic case) and parabolic case $(P)$ with $T=0.1$, $1$, and $5$. Figure \ref{fig:result-constant_result} shows the optimized distribution of $\phi$ and convergence history of the objective functionals $ J_\varepsilon^T(\phi)$ and $ J_\varepsilon^\infty(\phi)$. The results suggest that the optimal values approach that of $(P_\infty)$ monotonically as $T$ increases, which is consistent with the consequence of Theorem \ref{T:main} (ii). Indeed, we computed the value of $\lambda_1$, the smallest eigenvalue of $-\mathrm{div} (\kappa[(\phi_T^\ast)_+^m])$, and confirmed that \eqref{eq:f-u0-assump} is fulfilled. This ensures that the optimized configurations satisfy a necessary optimality condition in Remark \ref{R:interpre} (i).

\begin{figure}
    \centering
    \includegraphics[width=1\linewidth]{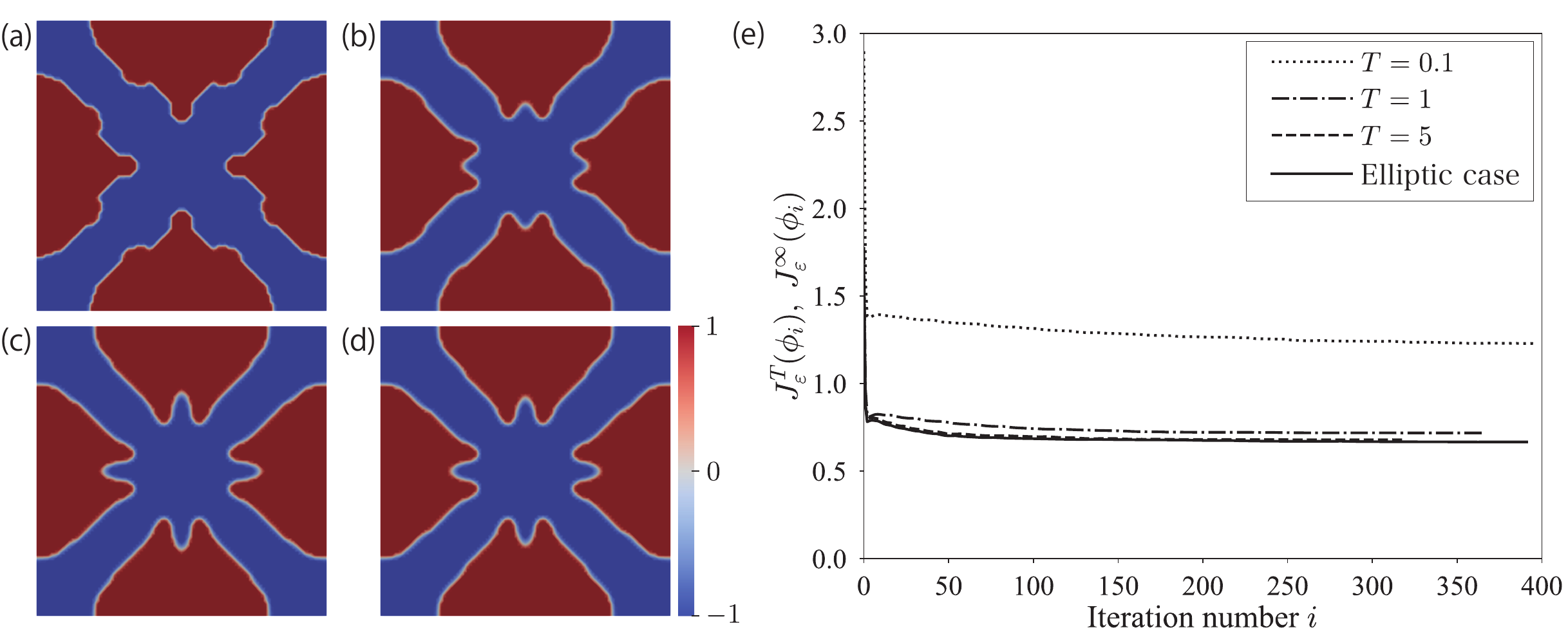}
    \caption{Optimized designs and convergence history of the objective functionals in the case of $f\equiv 10$. (a)--(d) show optimized level set functions $\phi$: (a) $T=0.1$ (b) $T=1$ (c) $T=5$ (d) elliptic case. The convergence histories of $J^T_\varepsilon(\phi)$ and {$J^\infty_\varepsilon(\phi)$} are plotted in (e).
    }
    \label{fig:result-constant_result}
\end{figure}

\begin{figure}
    \centering
    \includegraphics[width=1\linewidth]{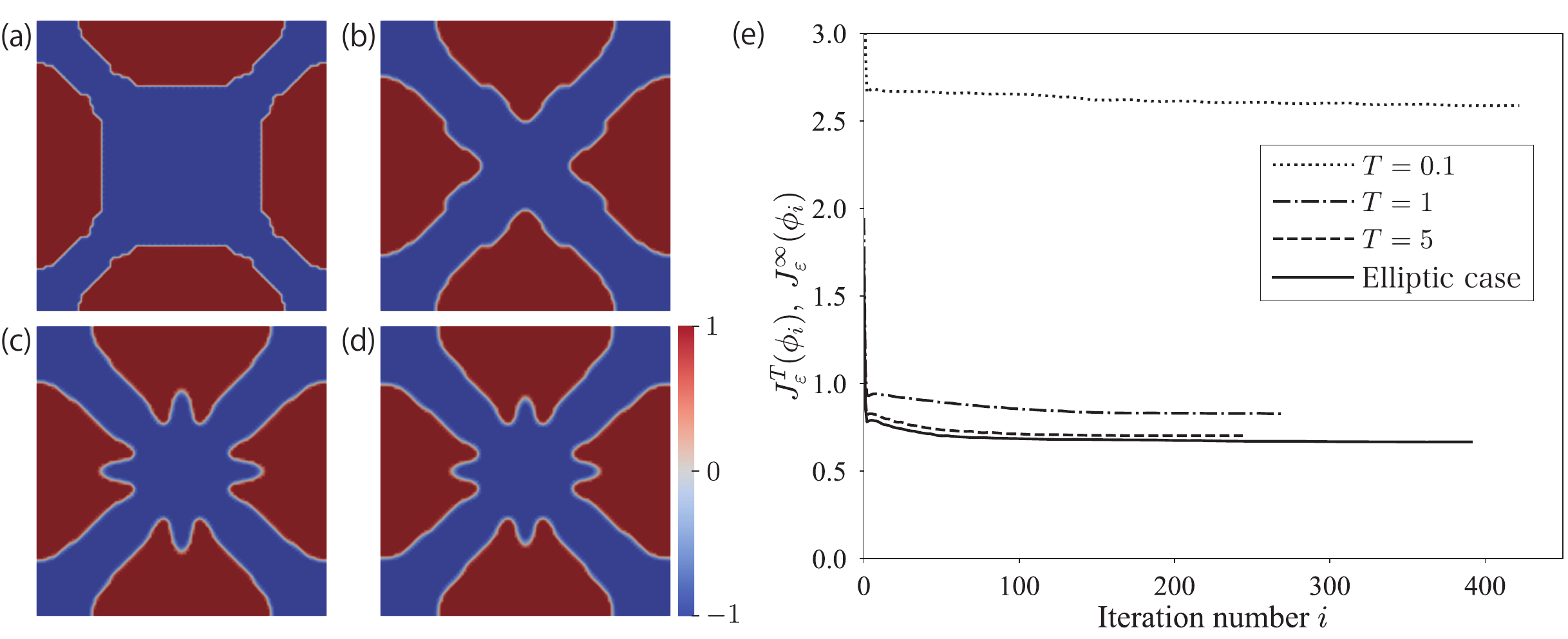}
    \caption{Optimized designs and convergence history of the objective functionals when the source $f$ is given by \eqref{eq:f-example}. (a)--(d) show optimized level set functions $\phi$: (a) $T=0.1$ (b) $T=1$ (c) $T=5$ (d) elliptic case. The convergence histories of $J_\varepsilon^T(\phi)$ and $J_\varepsilon^\infty(\phi)$ are plotted in (e).}
    \label{fig:result-varying_result} 
\end{figure}
We observed analogous numerical results for non-constant source $f$. 
As in the condition on $f$ for obtaining \eqref{eq:f-sym},
let us next consider the following function:
\begin{align}
    f(x,t) =
    \begin{cases}
        10(\exp(-\pi^2 t/2)\cos 4\pi^2 t+1), & t \leq T/2, \ x\in\Omega,
        \\
        f(x,T-t), & t > T/2, \ x\in\Omega.
    \end{cases}
    \label{eq:f-example}
\end{align}
It is easy to verify that \eqref{eq:f-example} satisfies the assumption of Theorem \ref{T:main} (ii). The same algorithm and parameters produce the results shown in Figure \ref{fig:result-varying_result}. These results suggest the convergence of the optimal values is still valid for time-dependent sources. It is worth noticing that not only the objective values but also the optimized level set functions $\phi_T^\ast$ asymptotically approach $\phi_\infty^\ast$, a minimizer of $(P_\infty)$, for large $T$.

\appendix
\section{Differentiability of the state variable}
\begin{lem}\label{L:u'}
For every $(\theta,A)\in \mathcal{RD}$, let $u\in V$ be a unique weak solution to \eqref{eq:1}--\eqref{eq:3} with $\kappa[\chi]$ being replaced by $A$. Then $A\mapsto u=u_A$ is Fr\'echet differentiable at $A$
for the direction $h\in L^{\infty}(\Omega;\mathbb{S}^{N\times N})$ such that $A+h$ satisfies the uniform ellipticity.
Furthermore,
$\tilde{u}:=u'h$ is characterized as a unique weak solution to the following \eqref{eq:u'1}--\eqref{eq:u'3}\/{\rm :}
\begin{align}
\partial_t \tilde{u}&=\dv(A\nabla \tilde{u}+h\nabla u)\quad &&\text{ in }\ \Omega\times (0,T), \label{eq:u'1}\\
\tilde{u}&=0 &&\text{ on }\ \partial \Omega\times (0,T),\label{eq:u'2}\\
\tilde{u}&=0 &&\text{ in }\ \Omega\times \{0\}.
\label{eq:u'3}
\end{align}
\end{lem}

\begin{proof}
Let $A[s]=A+sh$ for $s\in [0,1]$ and let $u[s]$ be a weak solution to \eqref{eq:1}--\eqref{eq:3} with $\kappa[\chi]$ being replaced by $A[s]$.
Differentiating the weak form of $u[s]$ by $s\in [0,1]$, we have, for any $\psi\in H:=L^2(0,T;H^1_0(\Omega))$, 
\begin{align*}
\int_0^T\langle \partial_t u'[s](t),\psi(t)\rangle_{H^1_0(\Omega)}\, \d t
=-
\int_0^T\int_{\Omega}[A[s](x)\nabla u'[s](x,t)+h(x)\nabla u[s](x,t)]\cdot \nabla \psi(x,t)\, \d x\d t,
\end{align*}
which coincides with the weak form of \eqref{eq:u'1}--\eqref{eq:u'3} as $\tilde{u}=u'[s]$ and $s=0$.
Noting that $A=A[0]=A[1]-h$, $u_{A+h}=u[1]$ and $u_{A}=u[0]$, we observe that, for any $\psi\in H$,
\begin{align}
 \lefteqn{  
 \int_0^T\langle \partial_t[u_{A+h}(t)-u_A(t)-\tilde{u}(t)], \psi(t)\rangle_{H^1_0(\Omega)}\, \d t}\nonumber\\
 &\quad +
 \int_0^T\int_{\Omega}A(x)\nabla[u_{A+h}(x,t)-u_A(x,t)-\tilde{u}(x,t)]\cdot \nabla \psi(x,t)\, \d x\d t
 \displaybreak[0]\nonumber\\
&=
 \int_0^T\langle \partial_t[u[1](t)-u[0](t)-\tilde{u}(t)], \psi(t)\rangle_{H^1_0(\Omega)}\, \d t\nonumber\\
 &\quad +
\int_0^T\int_{\Omega}[(A[1](x)-h(x))\nabla u[1](x,t)-A[0](x)\nabla u[0](x,t)-A(x)\nabla \tilde{u}(x,t)]\cdot \nabla \psi(x,t)\, \d x \d t
\displaybreak[0]\nonumber\\
& =
    -
   \int_0^T\int_{\Omega}h(x)\nabla (u[1](x,t)-u[0](x,t))\cdot \nabla \psi(x,t)\, \d x\d t
\le
    \|h\|_{L^{\infty}(\Omega)}\|u_{A+h}-u_A\|_{H}\|\psi\|_{H}.
 \nonumber
\end{align}
Here we claim that
\begin{align*}
\|u_{A+h}-u_A\|_{H}
\le C\|h\|_{L^{\infty}(\Omega)}.
\end{align*}
Indeed, by the same argument mentioned above, one can obtain
\begin{align*}
\lefteqn{\|u_{A+h}(T)-u_A(T)\|_{L^2(\Omega)}^2+\|
u_{A+h}-u_A\|_{H}^2}\\
&\le
    C\|h\|_{L^{\infty}(\Omega)}\|u_{A+h}\|_{H}\|u_{A+h}-u_A\|_{H}\\
&\le  C\|h\|_{L^{\infty}(\Omega)}(\|f\|_{L^2(0,T;H^{-1}(\Omega))}+\|u_0\|_{L^2(\Omega)})\|u_{A+h}-u_A\|_{H}.
\end{align*}
Here we used the fact that $A+h$ satisfies the uniform ellipticity to obtain $\|u_{A+h}\|_{H}\le C(\|f\|_{L^2(0,T;H^{-1}(\Omega))}+\|u_0\|_{L^2(\Omega)})$ in the last line.
Hence setting $\psi=u_{A+h}-u_A-\tilde{u}$ yields 
\begin{align*}
\|u_{A+h}(T)-u_A(T)-\tilde{u}(T)\|_{L^2(\Omega)}^2+
    \|u_{A+h}-u_A-\tilde{u}\|_{H}^2\le C\|h\|_{L^{\infty}(\Omega)}^2\|u_{A+h}-u_A-\tilde{u}\|_{H},
\end{align*}
which implies that
\begin{align}
    \lim_{\|h\|_{L^{\infty}(\Omega)}\to 0_+}\frac{\|u_{A+h}-u_A-\tilde{u}\|_{H}}{\|h\|_{L^{\infty}(\Omega)}}=0.
    \label{eq:A-1}
\end{align}
Furthermore, we deduce from the same argument that
$$
\|\partial_t(u_{A+h}(t)-u_A(t)-\tilde{u}(t))\|_{H^{-1}(\Omega)}
\le
\|h\|_{L^\infty(\Omega)}\|u_{A+h}(t)-u_A(t)\|_{H^1_0(\Omega)}.
$$
Hence it holds that
$$
\|\partial_t(u_{A+h}-u_A-\tilde{u})\|_{H^\ast}
\le
C
\|h\|_{L^\infty(\Omega)}^2,
$$
which together with \eqref{eq:A-1} yields 
$$
\lim_{\|h\|_{L^{\infty}(\Omega)}\to 0_+}\frac{\|u_{A+h}-u_A-\tilde{u}\|_{V}}{\|h\|_{L^{\infty}(\Omega)}}=0.
$$
This completes the proof.

\end{proof}

\section*{Declaration of competing interest}
The authors declare that they have no known competing financial interests or personal relationships that could have appeared to influence the work reported in this paper.

\section*{Data availability}
Data will be made available on request.

\end{document}